\documentclass[11pt]{article}
\pdfoutput=1
\usepackage{amssymb}
\usepackage{amsmath}
\usepackage{amsthm}
\usepackage{amsfonts}
\usepackage{enumerate}
\usepackage{url}
\usepackage{graphicx}

\newtheorem{theorem}{Theorem}[section]
\newtheorem{conjecture}[theorem]{Conjecture}

\newtheorem{lemma}[theorem]{Lemma}

\newtheorem{hypothesis}[theorem]{Hypothesis}

\DeclareMathOperator{\lcm}{lcm}
\DeclareMathOperator{\rad}{rad}

\begin{document}

\title{Are the Collatz and abc conjectures related?}
\author{Olivier Rozier}
\date{}

\maketitle

\begin{abstract}
The Collatz and $abc$ conjectures, both well known and thoroughly studied, appear to be largely unrelated at first sight. We show that assuming the $abc$ conjecture true is helpful to improve the lower bound of integers initiating a particular type of Collatz sequences, namely finite sequences of a given length where all terms but one are odd with the usual ``shortcut'' form. To obtain sharper bounds in this context, we are led to consider a small subset of the $abc$-hits. Then, it turns out that Collatz iterations as well as Wieferich primes may be used to find large triples in this subset.
\end{abstract}

\section{Introduction}

The Collatz conjecture deals with positive integer sequences arising when repeatedly applying the function
\begin{equation}
T(n) = \left\{\begin{array}{ll}
  \frac{3n+1}{2}, & \mbox{if n is odd,} \\
  \frac{n}{2}, & \mbox{otherwise.} 
\end{array}\right.
\end{equation}
Originally, Lothar Collatz introduced those sequences with the transformation $n \mapsto 3n+1$ for odd $n$, without dividing by 2 in the same step \cite{Lag10}. We say that $T$ has a shortcut, since the number of steps is smaller for the same final result. It is expected that whatever the first term of the sequence, it always ends up reaching the trivial cycle $(1,2)$. The latter assumption, called Collatz conjecture, is very popular under various names like the $3x+1$ or Syracuse problem, even outside the mathematical community.

On the other hand, the $abc$ conjecture of J. Oesterl\'e and D. Masser is considered as particularly important for its numerous implications in number theory \cite{Oes88,Wal03}. It poses a limitation on the presence of high powers of primes in Diophantine equations over three terms. There are various ways, more or less equivalent, to formulate this conjecture, which can be stated as follows:

\begin{conjecture}{\em ($abc$ conjecture)}\label{conj_abc}
For every $\varepsilon > 0$, there exist only finitely many triples $(a, b, c)$ of coprime positive integers satisfying $a + b = c$ and such that
\begin{equation}\label{eq:abc}
c > \rad(abc)^{1+\varepsilon }
\end{equation}
where $\rad(abc)$ denotes the radical of $abc$, that is the product of its distinct prime factors.
\end{conjecture}

The purpose of the present article is to highlight a possible relationship between the $abc$ and Collatz conjectures. Apparently, no direct relationship is mentioned in the literature, despite a previous attempt\footnote{In fact, Kaneda tried to apply the $n$-conjecture, which is a generalization of $abc$ with $n$ terms, first introduced in \cite{Bro94}.} by M. Kaneda briefly related at the end of \cite{Kan14}. Let us recall that the $abc$ conjecture originates in the theory of elliptic curves, an active field of research in number theory far from the Collatz conjecture \cite{Oes88}.

Both conjectures are nevertheless connected to linear forms in logarithms and, thus, to Baker theorem \cite{Bak75}. This number theoretic result is central when studying the existence of non-trivial Collatz cycles \cite{Sim05,Ste77} and has a clear relationship with numerous Diophantine equations linked to the $abc$ conjecture, e.g., Fermat last theorem \cite{Gra02,Oes88,Wal03}. Let us also point out that the set of Wieferich primes (namely, primes $p$ such that $2^{p-1}-1$ is a  multiple of $p^2$) is another example of a topic that is somehow connected to both conjectures. On the one hand, the existence of infinitely many non-Wieferich primes, a question which is still open despite numerical evidence, is implied by the  $abc$ conjecture \cite{Sil88}. On the other hand, the set of Wieferich numbers, a natural extension of the aforementioned set of primes, is at the heart of a partial solution to several variants of the Collatz conjecture \cite{Fra95}. Specifically, the existence of sequences not leading to 1 is established in the so-called ``$qx+1$'' variants for $q$ a Wieferich number, whereas the existence of sequences that diverge to infinity remains elusive for every odd $q>3$.

We start our study with a lower bound hypothesis (LBH) regarding the first term of Collatz sequences of finite length having a given number of odd and even terms. If proved true, such a statement is strong enough to settle the Collatz conjecture (Lemma 3.1 in \cite{Roz17}). Therefore, any progress regarding its validity might be of interest.

\begin{hypothesis} \label{hyp:LBH} $\emph{(Lower Bound Hypothesis - LBH)}$
There is a real constant $C \geq 0$ such that for all positive integers $j$ and $n$ not both equal to 1, we have
\begin{equation} \label{eq:LBH} 
 n \geq j^{-C} \; 2^{\left( 1-H\left( \frac{q}{j}\right)  \right) j}
\end{equation}
where $q$ is the number of odd integers in the vector $\left(  n, T(n), \ldots, T^{j-1}(n) \right)$ of the first $j$ iterates starting from $n$, and $H$ is the binary entropy function $H(x) = -x \log_{2}(x) - (1-x) \log_{2}(1-x)$ for $0<x<1$, with $H(0) = H(1) = 0$.
\end{hypothesis}

In other words, it is stated that integers $n$ whose parity vector 
$$\left(  n \mod 2, \ldots, T^{j-1}(n) \mod 2 \right)$$
has low binary entropy are unlikely to be found below a certain lower bound. This conjecture is known to hold true when $2q \leq j$ and when $q=j$.

Like the $abc$ conjecture, LBH is derived from heuristic considerations and supported by empirical data. In fact, both are linked to a particular representation of positive integers, their prime factorization for $abc$ and their parity sequence for LBH \cite{Roz19}. Those representations may be viewed as a kind of information whose entropy is measurable and can be translated into probabilistic terms, then enabling accurate predictions.

First, we show in Section \ref{sec:LB} that the $abc$ conjecture implies a lower bound reasonably close to \eqref{eq:LBH} in the particular case of sequences where all terms but one are odd ($q=j-1$). Then, in Section \ref{sec:mu_abc}, we take a closer look at the $abc$ conjecture and study the set ${\cal H}_{\mu}$ of triples $(a,b,c)$ of coprime positive integers such that $a+b=c$ and $\log c > \mu (abc)$ where
$$ \mu \left( {p_1}^{n_1} \ldots {p_k}^{n_k} \right) = \log (p_1 \ldots p_k) + \log (n_1 \ldots n_k)$$
for distinct primes $p_1, \ldots, p_k$ and positive integers $n_1, \ldots, n_k$. It turns out that ${\cal H}_{\mu}$ is a very sparse subset of the $abc$-hits. In Section \ref{sec:mu_abc_Collatz}, we prove that the lower bound \eqref{eq:LBH} holds with $C=1$ for the same Collatz sequences as in \S\ref{sec:LB}, unless maybe if we encounter an element of ${\cal H}_{\mu}$. Finally, in Section \ref{sec:Wief}, we show that the question of whether ${\cal H}_{\mu}$ is infinite has a close connection with the number of Wieferich primes.

\section{A particular type of Collatz sequences}
\label{sec:LB}

For any integer $j \geq 2$, let us consider Collatz sequences of length $j$ for which all terms but one are odd. Let ${\cal N}(j)$ denote the set of positive integers $n$ lower than $2^j$ initiating such sequences, i.e., for which there is exactly one even term among $n$, $T(n)$, \ldots, $T^{j-1}(n)$. The set ${\cal N}(j)$ contains exactly $j$ elements. This is a simple consequence of the one-to-one correspondence, discovered by C. J. Everett \cite{Eve77} and by R. Terras \cite{Ter76}, between the congruence classes modulo $2^j$ and the set of parity vectors of length $j$. E.g., one has $${\cal N}(10) = \lbrace 159, 239, 447, 511, 639, 681, 767, 795, 871, 1\,022 \rbrace.$$

We have shown in Lemma 6.1 of \cite{Roz17} that 
\begin{equation} \label{eq:lower_bound_nj}
n \geq 2^{j/(1+\rho)} - 2 \quad \text{for any $n \in {\cal N}(j)$}
\end{equation}
with $\rho = \log_2 3 = 1.585\ldots$. Unfortunately, this result is far from the desired lower bound given by LBH in the particular case $q=j-1$, which asserts that
\begin{equation} \label{eq:lower_bound_nj_lbh}
 n \geq j^{-C} \; 2^{\left( 1-H\left( 1-\frac{1}{j}\right)  \right) j} \quad \text{for any $n \in {\cal N}(j)$}
\end{equation}
where the right-hand side grows asymptotically as $j^{-(C+1)} \; 2^j$ up to a multiplicative constant.

Although LBH remains unsettled for this particular case, it can be helpful to specify which lower bound we should expect to hold. Yet, it is unclear which value of $C$ should be considered. Figure \ref{fig:LBH_misses} shows that, for each $ 2 \leq j \leq 3 \cdot 10^4$, the inequality \eqref{eq:lower_bound_nj_lbh} is satisfied with $C=1$ whereas it is increasingly falsified when $C$ decreases below $0.9$, suggesting that it is sharp with $C \approx 1$.

\begin{figure}[ht]
\centering
\includegraphics[width=0.8\textwidth]{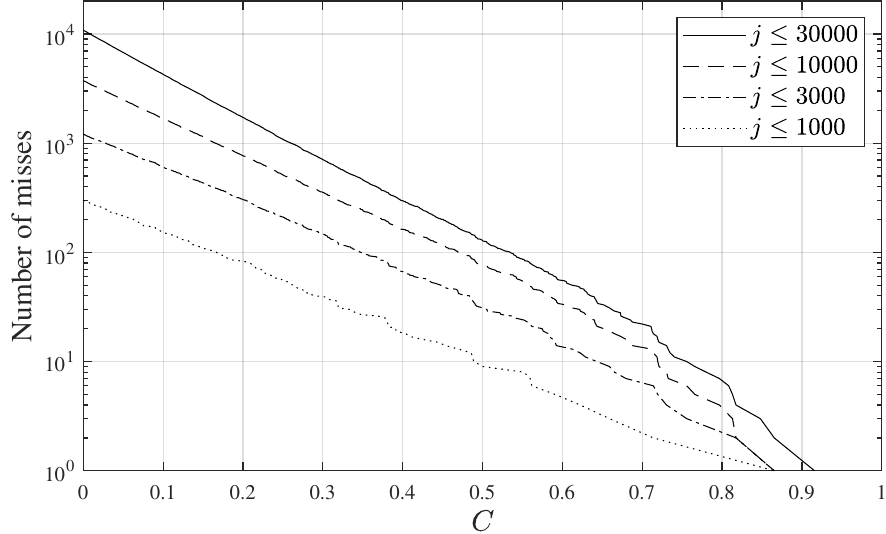}
\caption{Log plot of the number of misses of Hypothesis \ref{hyp:LBH} (LBH) with respect to the constant $C$ arbitrarily fixed between  0 and 1, in the particular case of Collatz sequences of length $j$ having exactly one even term ($q=j-1$) for $j$ up to 1\,000, 3\,000, 10\,000 and 30\,000.}
\label{fig:LBH_misses}
\end{figure}

If we assume the $abc$ conjecture true, we obtain a lower bound which is much better  than \eqref{eq:lower_bound_nj} and comes close to \eqref{eq:lower_bound_nj_lbh}.

\begin{theorem}\label{th:LBH_ABC}
Assume that the $abc$ conjecture is true. Then, for every $\varepsilon > 0$, there exists a  constant $K(\varepsilon) > 0$ such that 
\begin{equation} \label{eq:lower_bound_nj_abc_th}
n \geq K(\varepsilon) \; 2^{(1 - \varepsilon)j} \quad \text{for all $j \geq 2$ and $n \in {\cal N}(j)$}.
\end{equation}
\end{theorem}

Let us recall that the elements of ${\cal N}(j)$ are easy to compute by using, for each $0 \leq k \leq j-1$, the congruence 
\begin{equation} \label{eq:cong_nj}
n \equiv -1 - \left( \frac{2}{3} \right)^k \pmod{2^j}
\end{equation}
given in \cite[p4]{Roz19}, where $n$ is the element of ${\cal N}(j)$ such that $T^{k}(n)$ is even. Indeed, we have the congruences implied by \eqref{eq:cong_nj}
\begin{itemize}
\item $T^{i}(n) \equiv -1 - \left( \frac{2}{3} \right)^{k-i} \pmod{2^{j-i}}$ for any $0 \leq i \leq k$,
\item $T^{i}(n) \equiv -1 \pmod{2^{j-i}}$ as long as $k < i < j$,
\end{itemize}
so that all iterates $T^{i}(n)$, $i<j$, are odd except when $i=k$.

It is not difficult to see that Theorem \ref{th:LBH_ABC} follows directly from the next lemma.

\begin{lemma}\label{lem:LBH_ABC}
Assume that the $abc$ conjecture is true. Then, for every $\varepsilon > 0$ and every integer $j$ sufficiently large (i.e., for any $j \geq j_{\varepsilon}$ with $j_{\varepsilon}$ depending only on $\varepsilon$), we have the lower bound
\begin{equation} \label{eq:lower_bound_nj_abc}
n \geq \frac{1}{6} \; 2^{(1 - \rho \, \varepsilon)j} - 1 \quad \text{for any $n \in {\cal N}(j)$}
\end{equation}
with $\rho = \log_2 3$.
\end{lemma}

\begin{proof}
Let $j \geq 2$ and let $n \in {\cal N}(j)$ for which $T^{k}(n)$ is even with $0 \leq k \leq j-1$. 
Multiplying \eqref{eq:cong_nj} by $3^k$, we obtain
$$ 2^k + 3^k (n+1) = 2^j B$$
with $B$ a positive integer. It turns out than $n+1$ is divisible by $2^k$. Thus, we can write $n+1 = 2^k A$, which leads us to the equation
$$ 1 + 3^k A = 2^{j-k} B$$
with $3^k  A$ and $2^{j-k} B$ relatively prime. Then, we apply the $abc$ conjecture, assumed to be true, to the above triple for an arbitrary $\varepsilon > 0$. This gives the lower bound
$$  2^{j-k} B \leq \rad\left( 2^{j-k} \, 3^k \, A B \right)^{1+\varepsilon},$$ 
except maybe for finitely many cases that can be ruled out if we assume $j$ sufficiently large, that is $j \geq j_{\varepsilon}$ for a suitable $j_{\varepsilon}$. Since $\rad\left( 2^{j-k} \, 3^k \, A B \right) \leq \rad(6 A B) \leq 6 A B$, we get
$$ 2^{j-k} \leq \left( 6 A \right)^{1+\varepsilon}  B^\varepsilon$$
or, by taking the logarithm in base 2,
$$ j-k \leq (1+\varepsilon)  \log_2 (6A) + \varepsilon \log_2 B,$$
so that
$$ \log_2 6(n+1) = k + \log_2 (6A) \geq k + \frac{j - k - \varepsilon \log_2 B}{1+\varepsilon}.$$
Note that $n \leq 2^j-2$, as $2^j$ and $2^j-1$ are not in ${\cal N}(j)$. Hence,
\begin{align*}
B = \frac{3^k}{2^j} (n+1) + 2^{k-j} &\leq \frac{3^k}{2^j} \left( 2^j -1\right) + 2^{k-j} \\
&= 3^k - \frac{3^k - 2^k}{2^j} \\
&\leq 3^k.
\end{align*}
Putting together all inequalities yields 
\begin{align*}
 \log_2 6(n+1) &\geq k + \frac{j - k - \varepsilon \rho k}{1+\varepsilon} \\
 &= \frac{j - \varepsilon(\rho-1)k}{1 + \varepsilon} \\
 &> \left( \frac{1 - \varepsilon(\rho-1)}{1 + \varepsilon} \right) j \\
 &> (1 - \varepsilon \rho) j.
\end{align*}
\end{proof}

\begin{proof}[Proof of Theorem \ref{th:LBH_ABC}]
Let us fix $\varepsilon > 0$. We can further assume that $\varepsilon < 1$, otherwise \eqref{eq:lower_bound_nj_abc_th} is trivially satisfied for $K(\varepsilon) = 1$. 
In Lemma \ref{lem:LBH_ABC}, one may substitute $\frac{\varepsilon}{\rho}$ for $\varepsilon$ and we obtain that, for some $j_0$ depending on $\varepsilon$,
$$n \geq \frac{1}{6} \; 2^{(1 - \varepsilon)j} - 1 \quad \text{for any $j \geq j_0$ and $n \in {\cal N}(j)$}.$$
Now, we choose an integer $j_1 \geq 3$ such that 
$$ 2^{-(1-\varepsilon)j_1} \leq \frac16 - K_1  \quad \text{for some $0 < K_1 < \frac16$}.$$
If we set $j_2 = \max(j_0,j_1)$, then we have
$$ n \geq K_1 \; 2^{(1 - \varepsilon)j} \quad \text{for any $j \geq j_2$ and $n \in {\cal N}(j)$}.$$
Next, choose a constant $K_2 > 0$ such that 
$$ n \geq K_2 \; 2^{(1 - \varepsilon)j} \quad \text{for any $1 < j < j_2$ and $n \in {\cal N}(j)$}.$$
This is always possible due to the finiteness of ${\cal N}(j)$. A suboptimal choice might be $K_2 = 2^{-(1-\varepsilon)j_2}$, which is sufficient for the intended purpose although it could be substantially improved by using \eqref{eq:lower_bound_nj}. Finally, we obtain the lower bound \eqref{eq:lower_bound_nj_abc_th} for the suitable constant $K(\varepsilon) = \min ( K_1, K_2 ).$
\end{proof}

One may object that the lower bound \eqref{eq:lower_bound_nj_abc_th} derived from the $abc$ conjecture is still weaker than the lower bound \eqref{eq:lower_bound_nj_lbh} derived from LBH. To help resolve this discrepancy, we propose in the next section to dive deeper into the $abc$ conjecture and refine the notion of $abc$-hit.

\section{A rare type of $abc$-hits}
\label{sec:mu_abc}

There exist infinitely many triples $(a,b,c)$ of coprime positive integers such that $a+b=c$ and $c > \rad(abc)$, which are called $abc$-hits or $abc$ triples. To quantify how much $\rad(abc)$ differs from $c$, one generally refers to the notion of \textit{quality} defined by
$$q(a,b,c) = \frac{\log c}{\log \rad(abc)},$$
which is a real number greater than 1 for any $abc$-hit, hereafter noted $q$. The $abc$ conjecture asserts that, for every $\varepsilon > 0$, there exist only finitely many $abc$-hits with $q > 1+\varepsilon$.
For instance, the $abc$-hits of quality $q > 1.4$, assumed to be finite in number, are often called ``good'' $abc$ triples. Their list is maintained on \cite{ABC} and contains 241 good $abc$ triples so far. Several algorithms have been proposed to search for triples of high quality, e.g., by using the continuous fraction expansion of algebraic numbers \cite{Bro94,Nit93}.

Somehow, the quality $q$ is not a metric well-suited to treat equations of the form
$$ 1 + 3^k A = 2^l B,$$
which are at the heart of the proof of Theorem \ref{th:LBH_ABC}. Indeed, the quality takes into account the size of the radicals, but the size of the exponents is poorly constrained.

We suggest replacing the radical by a different function where the size of the exponents also comes into play. Thus, let us introduce the function $\mu : \mathbb{Z}_{\geq 1} \rightarrow \mathbb{R}_{\geq 0} $ uniquely defined by two properties:
\begin{enumerate}[(i)]
\item $\mu(p^n) = \log p + \log n$, for prime $p$ and positive integer $n$;
\item $\mu(m n) = \mu(m) + \mu(n)$, for coprime positive integers $m,n$.
\end{enumerate}
It follows from (ii) that $\mu(1) = 0$. An equivalent definition of $\mu$ is 
\begin{equation}
\mu(n) = \log \rad (n) + \log \prod_{p | n} \nu_p(n)
\end{equation}
where the product runs over all prime divisors of $n$ and where $\nu_p$ stands for the $p$-adic valuation.

The function $\mu$ is a crude measure of the number of digits (i.e., the amount of information) that are necessary to write the prime factorization of a positive integer, independently of the base of a given numeral system.
One may apply a multiplicative constant to adjust this measure to a particular base. Recall that the amount of information of the leading digit should be weighted non-uniformly, according to Benford's law. To this respect, the measure $\mu$ is consistent with Benford's law.
Furthermore, it satisfies the inequalities 
\begin{equation}\label{eq:mu_bounds}
\log \rad (n) \leq \mu (n) \leq \log n,
\end{equation}
and
\begin{equation}\label{eq:mu_prod}
\max(\mu (m), \mu (n)) \leq \mu (mn) \leq \mu (m) + \mu (n)
\end{equation}
for all positive integers $m,n$. The properties \eqref{eq:mu_bounds} and \eqref{eq:mu_prod} result from the well-known inequalities $a+b \leq ab \leq a^{b}$ for $a,b \geq 2$. If $n$ is square-free, a double equality holds in \eqref{eq:mu_bounds}. One may interpret these properties in terms of the data compression that occurs occasionally when performing the prime factorization of integers. This compression is lossless and reversible, whereas taking the radical is not.

By analogy with the $abc$-hits, we define the set 
$${\cal H}_{\mu} = \left\lbrace (a,b,c) \in \mathbb{Z}_{\geq 1}^3 :\; a+b=c,\, \gcd(a,b)=1 \text{ and } \log c > \mu(abc) \right\rbrace$$
and call \textit{$\mu$-hits} the elements of ${\cal H}_{\mu}$. As a result of the first inequality in \eqref{eq:mu_bounds}, the set ${\cal H}_{\mu}$  is trivially a subset of the $abc$-hits. It is also not empty as it contains the triple $\left( 1, 239^2, 2 \cdot 13^4 \right) $, which is the only $\mu$-hit below one million.

To estimate the size of ${\cal H}_{\mu}$, we analysed the exhaustive list of all the $abc$-hits below $10^{18}$ from the project ABC@home \cite{ABC}. Among these $14\,482\,065$ triples, we found only 464 $\mu$-hits\footnote{The list of $\mu$-hits below $10^{18}$ is available at \url{https://www.ipgp.fr/~rozier/abc/mu_hits_1018.txt}.}, of which 56 are also good $abc$ triples. Additionally, we found 175 $\mu$-hits among about 9 million $abc$-hits with $c$ between $10^{18}$ and $2^{63} \approx 9.22\cdot 10^{18}$ also identified by ABC@home.

In view of Figure \ref{fig:dist_triples}, it appears that ${\cal H}_{\mu}$ contains an increasingly small proportion of the $abc$-hits as their size increases, and only a fraction of the good $abc$ triples. Empirically, the number of $\mu$-hits below $x$ tends to follow a power law of the form $\left(x / x_0 \right)^{\alpha}$ with $\alpha \approx \frac{2}{11}$ and $x_0 \approx 10^3$, unlike the good $abc$ triples whose number is expected to be upper bounded according to the $abc$ conjecture.

\begin{figure}[ht]
\centering
\includegraphics[width=0.8\textwidth]{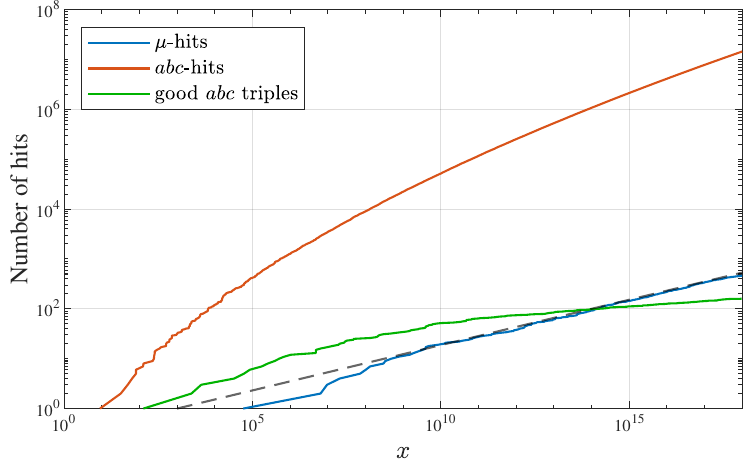}
\caption{Log-log plot of the number $N$ of $\mu$-hits (blue), $abc$-hits (red) and good $abc$ triples (green) below $x$. The dashed line corresponds to the power law $N = \left(x / x_0 \right)^{\alpha}$ with $\alpha = 2/11$ and $x_0=1\,000$. }
\label{fig:dist_triples}
\end{figure}

\renewcommand{\arraystretch}{1.2}
\setlength{\tabcolsep}{10pt}
\begin{table}[ht]
\begin{center}
\begin{tabular}{|cccrrr|}
  \hline
  a & b & c & digits & $q$ & $g$ \\
  \hline
  \hline
  1 & $239^2$ & $2 \cdot 13^4$ & 5 & 1.2540 & 0.139 \\
  2 & $3^{10} \cdot 109$ & $23^5$ & 7 & 1.6299 & 2.147 \\
  11 & $7^3 \cdot 167^2$ & $2 \cdot 3^{14}$ & 7 & 1.4283 & 0.389 \\
  1 & $3 \cdot 5^5 \cdot 47^2$ & $2^{18} \cdot 79$ & 8 & 1.4497 & 0.032 \\
  $2^2$ & $3^{15} \cdot 5$ & $17^4 \cdot 859$ & 8 & 1.3925 & 0.311 \\
 \hline
\end{tabular}
\caption{List of the five smallest $\mu$-hits $(a,b,c)$, number of digits of $c$, quality $q$ and gain $g$.}
\label{tab:mu_abc_hits}
\end{center}
\end{table}

Table \ref{tab:mu_abc_hits} gives the list of the five smallest $\mu$-hits, where the second one is the $abc$-hit with  the highest known quality discovered by E. Reyssat \cite{ABC,Bro94}. Much like the quality for $abc$-hits, it may be relevant to define a specific metric for $\mu$-hits. We propose to use
$$g(a,b,c) = \log c - \mu(abc)$$
which measures the gain of, say, ``digital information'' when expressing the prime factorization of $abc$ relatively to the standard writing of $c$. It is interesting to observe that about 75\% of the $\mu$-hits below $10^{18}$ have a gain $g <1$. The largest gain in this range is
$$ g\left( 19 \cdot 1\,307, 7 \cdot 29^2 \cdot 31^8, 2^8 \cdot 3^{22} \cdot 5^4\right) \approx 4.55 $$
from a triple originally discovered by J. Browkin and J. Brzezi\'nski \cite{Bro94}, ranked third by quality in \cite{ABC}. Moreover, when considering the list of 241 good $abc$ triples from \cite{ABC}, we find only 128 positive gains of which the largest is
$$ g\left( 2^2 \cdot 11, 3^2 \cdot 13^{10} \cdot 17 \cdot 151 \cdot 4\,423, 5^9 \cdot 139^6 \right) \approx 6.87 $$
from a triple ranked 15 by quality and discovered by A. Nitaj \cite{Nit93}.

First, one may ask if there exists a formula or an algorithm that gives infinitely many $\mu$-hits, as is the case for $abc$-hits. Indeed, we know formulas like
\begin{equation}\label{formula_abc_hits}
\left( 1, n^{2^k}-1, n^{2^k}\right) \quad \text{where $n<2^{k+1}$ is odd,}
\end{equation}
given in \cite{Lan90} with $n=3$ and for which every triple is a $abc$-hit. According to our computations, those formulas rarely generate $\mu$-hits. In other words, the gain of the resulting triples is most often negative, but not by far in the case of \eqref{formula_abc_hits} because $n^{2^k}-1$ is divisible by $2^{k+2}$. In fact, there seems to be no $\mu$-hit of this form for $n<239$. However, when setting $n=239$, we obtain a sequence of thirteen $\mu$-hits for $k=2,\ldots,14$ due to the divisibility of $n^{2^k}-1$ by $2^{k+4}$ and by $n^2+1 = 2 \cdot 13^4$ which turns out to be the smallest $\mu$-hit, as shown in Table \ref{tab:mu_abc_hits}. The gain $g$ is hard to compute with certainty for $k \geq 7$, but is fairly easy to estimate, assuming that no power of a large prime appears when factorizing the second term of the triple. Hence, for $k \geq 2$, we predict that the gain is close to the lower bound
$$\log \left( \frac{2 \cdot 13^3}{239}\right) - \log(k+4)$$
which is steadily decreasing and becomes negative for $k \geq 15$. It yields that there exist very large $\mu$-hits with as much as 38\,968 digits when setting $k=14$.

Next, we investigated this issue with other methods (e.g., the transfer method \cite{Mar16}) known to generate infinite sequences of $abc$-hits. But we found very little $\mu$-hits by using these methods.

As a result of the above observations, one must keep in mind that $\mu$-hits are far more difficult to generate than $abc$-hits.

\section{From Collatz sequences to $\mu$-hits}
\label{sec:mu_abc_Collatz}

In this section, we show how the notion of $\mu$-hit arises in the context of Collatz sequences. According to the following result, indeed, a lower bound stronger and more accurate than that of Theorem \ref{th:LBH_ABC} should hold true unless maybe if a $\mu$-hit of a given form is encountered.

\begin{theorem}\label{th:mu_ABC_LBH}
Let $T$ denote the Collatz function and let $j \geq 2$ be an integer. If $n$ is a positive integer such that the sequence $n$, $T(n)$, $\ldots$, $T^{j-1}(n)$ contains exactly one even term $T^k(n)$ with $k<j$, at least one of the following holds true:
\begin{enumerate}
\item We have the lower bound
\begin{equation}\label{eq:lower_bound_nj_mu_abc}
n > \frac{ 2^{j+1}}{3 \, j^{2}} - 1.
\end{equation}
\item There is a $\mu$-hit of the form $(1, b, b+1)$ with $b = T^k(n) + 1$.
\end{enumerate}
\end{theorem}

\begin{proof}
We may assume that $n$ is an element of ${\cal N}(j)$ as defined in \S\ref{sec:LB}, otherwise we have $n \geq 2^j$ and statement \textit{1} is satisfied. Let $k$ be the smallest integer for which $T^k(n)$ is even.

The case $k=0$ gives $n= 2^j-2$ by using \eqref{eq:cong_nj} and, again,  statement \textit{1} is satisfied. Hereafter, we assume that $1 \leq k \leq j-1$.

We follow the proof of Lemma \ref{lem:LBH_ABC} which leads to an equation of the form
$$ 1 + 3^k A = 2^{j-k} B $$
with $n = 2^k A - 1$ and $A$, $B$ positive integers. The first $k$ iterates are $ T^i(n) = 3^i 2^{k-i} A -1$ for $i \leq k$, so that
$$  T^k(n) = 3^k A -1.$$
Setting $b= T^k(n) + 1$, we ask whether the triple $(1, b, b+1)$ is in ${\cal H}_{\mu}$.
If it is in ${\cal H}_{\mu}$, we are led to statement \textit{2}. Otherwise, we have the inequality 
\begin{align*}
\log (b+1) &\leq \mu \left( b(b+1) \right) \\
 &= \mu(b) + \mu(b+1) \\
 &= \mu \left( 3^k A \right) + \mu \left( 2^{j-k} B \right) \\
 &\leq \mu \left( 3^k \right) + \mu \left( 2^{j-k} \right) + \mu (A) + \mu(B),  \quad \quad \quad \text{by using \eqref{eq:mu_prod},}\\
 &= \log 3 + \log k + \log 2 + \log (j-k) + \mu (A) + \mu(B) \\
 &\leq \log 6 + \log \left( k (j-k) \right) +\log A + \log B,  \quad \quad \text{by using \eqref{eq:mu_bounds}}.
\end{align*}
It yields that
\begin{align*}
 \log A &\geq \log (b+1) - \log B - \log \left( k (j-k) \right) - \log 6 \\
 &= \log \left( 2^{j-k} B \right) - \log B - \log \left( k (j-k) \right) - \log 6  \\
 &= (j-k) \log 2 - \log \left( k (j-k) \right) - \log 6, 
\end{align*}
so that
$$ n = 2^k A - 1 \geq \frac{2^j}{6 k (j-k)} - 1.$$
To conclude that statement \textit{1} holds true, observe that $k (j-k) \leq  \bigl( \frac{j}{2} \bigr) ^2$.
\end{proof}

Let us point out that Theorem \ref{th:mu_ABC_LBH} is of interest for two reasons. On the one hand, the effective bound \eqref{eq:lower_bound_nj_mu_abc} in statement \textit{1} fully agrees with the expected bound \eqref{eq:lower_bound_nj_lbh} derived from Hypothesis \ref{hyp:LBH} (LBH) when setting $C=1$. One may further verify that \eqref{eq:lower_bound_nj_mu_abc} is slightly stronger than \eqref{eq:lower_bound_nj_lbh} for $C=1$. On the other hand, Theorem \ref{th:mu_ABC_LBH} also provides a method based on the Collatz function for finding large $\mu$-hits.


\renewcommand{\arraystretch}{1.2}
\setlength{\tabcolsep}{7.5pt}
\begin{table}
\begin{center}
\begin{tabular}{|rrrccccr|}
  \hline
  j & k & digits & pow($A$) & pow($B$) & $q$ & $g$ & $C$ \\
  \hline
  \hline
  19 & 16 & 9 & - & $53^3$ & 1.474 & 1.313 & $-1.285$ \\
  85 & 56 & 32 & - & - &  1.115 & 0.097 & 0.865 \\
  108 & 85 & 47 & - & $7^4 \cdot 31^3$ & 1.137 & 2.861 & $-0.783$ \\
  160 & 26 & 53 & $7^8$ & - & 1.110 & 1.825 & $-1.151$ \\
  294 & 127 & 111 & $37^2$ & $7^6$ & 1.052 & $0.176$ & $-0.987$ \\
  626 & 382 & 251 & - & $17^2$ & 1.023 & 0.588 & 0.659 \\
  783 & 677 & 354 & - & $13^3 \cdot 43^4$ & 1.022 & 3.566 & $-0.758$ \\
  861 & 15 & 259 & $7^2$ & $5^2 \cdot 11^2$ & 1.022 & 1.098 & 0.105 \\
  874 & 45 & 271 & $7^9$ & - & 1.024 & 1.987 & $-1.009$ \\
  1\,056 & 921 & 479 & - & $47^3 \cdot 109^2$ & 1.013 & 0.858 & $-0.600$ \\
  1\,094 & 2 & 329 & $7^2 \cdot 13^2 \cdot 1\,093^2$ & ($B=1$) & 1.016 & 2.145 & $-0.829$ \\
  1\,357 & 1\,174 & 614 & $7^5$ & $5^4$ & 1.011 & 0.194 & $-0.494$ \\
  1\,367 & 296 & 463 & $7^2 \cdot 67^4$ & - & 1.015 & 0.697 & $-0.768$ \\
  1\,475 & 231 & 485 & - & $11^3 \cdot 13^3 \cdot 97^3$ & 1.016 & 1.469 & $-1.130$ \\
  2\,035 & 606 & 719 & - & $5^4 \cdot 11^6$ & 1.010 & 0.015 & $-0.890$ \\
  2\,186 & 2 & 658 & $7^2 \cdot 13^2 \cdot 1\,093^2$ & ($B=1$) & 1.008 & 1.452 & $-0.844$ \\
  3\,279 & 3 & 987 & $7^2 \cdot 13^2 \cdot 1\,093^2$ & ($B=1$) & 1.006 & 1.739 & $-0.716$ \\
  3\,514 & 4 & 1\,057 & $3\,511^2$ & ($B=1$) & 1.004 & 0.524 & $-0.584$ \\
  4\,370 & 2 & 1\,315 & $7^2 \cdot 13^2 \cdot 1\,093^2$ & ($B=1$) & 1.004 & 0.758 & $-0.857$ \\
  4\,393 & 436 & 1\,396 & - & $13^4 \cdot 23^2$ & 1.005 & 0.519 & $-0.175$ \\
  5\,461 & 488 & 1\,723 & $5^2$ & - & 1.004 & 1.024 & 0.813 \\
  6\,962 & 396 & 2\,159 & - & - & 1.003 & 0.063 & 0.766 \\
  13\,056 & 11\,808 & 6\,002 & - & $17^2$ & 1.001 & 1.284 & 0.735 \\
  27\,466 & 13\,732 & 10\,678 & - & $17^2$ & 1.001 & 1.874 & 0.916 \\
  28\,107 & 27\,795 & 13\,351 & $7^4$ & $5^3$ & 1.001 & 0.721 & 0.066 \\
 \hline
\end{tabular}
\caption{For $\mu$-hits of the form $\left( 1, 3^k A, 2^{j-k} B \right)$ related to Collatz, length $j$ of the corresponding sequence and index $k$ of its unique even term, number of digits of $2^{j-k} B$, prime powers in the factorization of $A$ and $B$, quality $q$ and gain $g$ of the triple, value of $C$ for which the equality holds in LBH on $j$ iterations.}
\label{tab:mu_abc_lbh}
\end{center}
\end{table}


In Theorem \ref{th:mu_ABC_LBH}, statement \textit{2} seems unlikely to occur, except maybe on rare occasions in view of the low density expectations regarding ${\cal H}_{\mu}$, so that statement 1 should hold true in most cases. This prediction is straightforward to verify for Collatz sequences of various lengths $j$ by computing the $j$ elements of ${\cal N}(j)$. In practice, we conducted a systematic search for all $j\leq 5\,000$. Due to the large size of the numbers considered, we used trial division to obtain a partial factorization with a size limit of $10^6$ on prime factors. Therefore, we possibly missed $\mu$-hits involving powers of large primes, although the probability is low. For all $5\,000 < j \leq 30\,000$, we restricted the search to sequences for which equality holds in LBH with $C \geq 0$ (see Figure \ref{fig:LBH_misses}). Such a restriction should not affect the verification of statement 1, but we probably missed many $\mu$-hits for which equality holds in LBH with $C<0$.

According to the computation results, inequality \eqref{eq:lower_bound_nj_mu_abc} is always satisfied in this range. Unexpectedly, we found a number of cases where statements \textit{1} and \textit{2} hold together. These cases are linked to a number of $\mu$-hits detailed in Table \ref{tab:mu_abc_lbh}. Remarkably, the smallest triple referenced in this table is $\left(1 , 3^{16} \cdot 7, 2^3 \cdot 11 \cdot 23 \cdot 53^3 \right)$, which is one of the $abc$-hit with the highest known quality (rank 22 by quality to date) discovered by A. Nitaj \cite{ABC,Bro94,Nit93}. Moreover, bound \eqref{eq:lower_bound_nj_mu_abc} is very sharp (by less than 1\%) for the second $\mu$-hit in Table \ref{tab:mu_abc_lbh}.

One may distinguish two kinds of $\mu$-hits in Table \ref{tab:mu_abc_lbh}, those for which \eqref{eq:lower_bound_nj_mu_abc} is rather sharp (that is, when the value $C$ in the last column of Table \ref{tab:mu_abc_lbh} is close to 1), and those more numerous for which high powers of primes occur in the factorization of $A$ and $B$.

Here, we should emphasize that there exist $\mu$-hits related to a Collatz sequence of  length $j$ starting above $2^j$, so that they are not listed in 
 Table \ref{tab:mu_abc_lbh}. E.g., the sequence starting from $n=13\,806\,249$ with length $j=19$ is linked to the fourth $\mu$-hit in Table \ref{tab:mu_abc_hits} discovered by G. Frey \cite{ABC,Bro94}.

It is amazing to observe that five $\mu$-hits (for which $B=1$) referenced in Table \ref{tab:mu_abc_lbh} are closely related to the known Wieferich primes, 1\,093 and 3\,511. They belong to a large family of $\mu$-hits of the form $\left( 1, b, b+1 \right)$ with
$$b = 2^{m(p-1)} - 1, \quad \text{$m \geq 1$ and $p$ a Wieferich prime.}$$
The associated Collatz sequences, of length $j=m(p-1)+k$ with $k$ the 3-adic valuation of $b$, start from the integer $n = (2^{j} - 3^k - 2^k)/3^k$, which is generally well above the bound \eqref{eq:lower_bound_nj_mu_abc}, then leading to $T^k(n) = 2^{j-k} - 2$ and to $T^j(n) = 3^{j-k-1} - 1$. In the next section, we briefly analyse the family of $\mu$-hits linked to Wieferich primes.

\section{Further insights from Wieferich primes}
\label{sec:Wief}

According to the computations of the previous section, all $\mu$-hits of the form $\left( 1, 2^n - 1, 2^{n} \right)$ seem to occur when $n$ is divisible by $p-1$ for some Wieferich prime $p$. Conversely, if $p$ is a Wieferich prime, one may wonder how many $\mu$-hits  there are of the form $(1,b_m,b_m+1)$ with $b_m = 2^{m(p-1)} - 1$ and $m$ a positive integer. Assuming that $(1, b_1,b_1 + 1)$ is a $\mu$-hit and taking into account the divisibility of $b_m$ by $b_1$ is not sufficient to ensure that the triple $(1,b_m,b_m+1)$ is also a $\mu$-hit for any $m$. Therefore, it is so far unclear if we can generate infinitely many $\mu$-hits of this form as $m$ increases. By setting successively $p=1\,093$ and $p=3\,511$, we find that the respective proportions of such $\mu$-hits with $m \leq 1\,000$ are at least 24\% and 27\%, by checking only the powers of small primes. Note that these proportions tend to decrease when considering larger values of $m$.
The next lemma shows that $\mu$-hits are far more likely to occur when combining several Wieferich primes.

\begin{lemma}\label{lem:mu_Wief}
Let $k \geq 1$ and let $p_1, \ldots, p_k$ be distinct Wieferich primes. If we put $L = \lcm \left( p_1 - 1, \ldots, p_k - 1\right) $, then we have
$$g\left( 1,2^L-1,2^L\right) > \log\left( \frac{p_1 \cdots p_k}{2^{k+1} L}\right).$$
\end{lemma}

\begin{proof}
Assuming that $p_1, \ldots, p_k$ are $k$ distinct Wieferich primes, then, for any $1 \leq i \leq k$, $2^{L}-1$ is divisible by $2^{p_i-1}-1$ and, consequently, by $p_i^2$. We infer that
\begin{align*}
\mu \left( 2^L - 1 \right)  &\leq \mu \left( p_1^2 \cdots p_k^2 \right) + \mu \left( \frac{2^L - 1}{ p_1^2 \cdots p_k^2}\right), \quad \text{by using \eqref{eq:mu_prod},}\\
 &< \mu \left( p_1^2 \cdots p_k^2 \right) + \log \left( \frac{2^L}{ p_1^2 \cdots p_k^2}\right), \quad \text{by using \eqref{eq:mu_bounds},}\\
 &= (k+L) \log 2 - \log \left( p_1 \cdots p_k \right),
\end{align*}
so that
\begin{align*}
g\left( 1,2^L-1,2^L\right) &> \log \left( p_1 \cdots p_k \right) - k \log 2 - \mu \left( 2^L \right) \\
 &= \log\left( \frac{p_1 \cdots p_k}{2^{k+1} L}\right).
\end{align*}
\end{proof}

As a consequence of Lemma \ref{lem:mu_Wief}, the gain of a triple of the form $\left( 1, 2^{p-1} - 1, 2^{p-1} \right)$ with $p$ a Wieferich prime is bounded from below by $-\log 4$, which is not enough to conclude that such a triple is a $\mu$-hit. However, if we apply this lemma to the known Wieferich primes $p_1=1\,093$, $p_2=3\,511$, and put $L = \lcm \left(  p_1 - 1, p_2 - 1 \right) = 49\,140$, we obtain the lower bound
$$g\left( 1, 2^L-1, 2^L \right) >  \log\left( \frac{p_1 p_2}{8 L}\right) = 2.278\ldots$$
without taking into account the divisibility of  $2^L-1$ by the prime powers $3^4$, $5^2$, $7^2$, and $13^2$. By doing so, we get the much higher bound
$$g\left( 1, 2^L-1, 2^L \right) \geq  8.228\ldots$$
which is likely to be the actual gain of this triple whose second element has 14\,793 digits. From this particular example, it is possible to generate thousands of large $\mu$-hits of the form $\left( 1,2^{mL}-1,2^{mL}\right) $ by choosing adequate values of $m$. E.g., when setting $m=q^k$ with $q$ an odd prime such that $q-1$ divides $L$ (that is, $q \in \left\lbrace 3,5,7,11,13,19,29,\ldots ,24\,571 \right\rbrace $), we expect the gain of the resulting triples to decrease very slowly as $k$ increases, due to the fact that $2^{q^k L}-1$ is divisible by $3^4 \cdot 5^2 \cdot 7^2 \cdot 13^2 \cdot 1\,093^2 \cdot 3\,511^2 \cdot q^k$, where $q^k$ can be replaced by $q^{k+1}$ whenever $2^L-1$ is not divisible by $q^2$. The proof by induction on $k$ is  straightforward.

Regarding the question of whether there exist infinitely many $\mu$-hits, we present a heuristic argument directly related to the Wieferich primes. The main point is that
this set of primes is expected to be infinite although only two are known, according to simple statistical considerations \cite{Cra97}. Indeed, one expects the number of primes $p$ below $x$ having presumably a given property with probability $1/p$ to grow like $\log \log x$, when applying the prime number theorem. Such a reasoning, whose prediction can hardly be neglected, was named the ``$\log \log$ Philosophy'' by Jean-Pierre Serre (see \cite[p413]{Rib96}).

\begin{theorem}\label{th:mu_Wief}
If the set of Wieferich primes is infinite, then so is the set ${\cal H}_{\mu}$ of $\mu$-hits.
\end{theorem}

\begin{proof}
First, we set $p_1=1\,093$, $p_2=3\,511$, and $L_2 = \lcm \left( p_1 - 1, p_2 - 1 \right) = 49\,140$. If we assume the set of Wieferich primes to be infinite, then there exist distinct Wieferich primes $p_3, \ldots, p_k$ not equal to $p_1,p_2$ for any $k \geq 3$. Setting $L_k = \lcm \left( p_1 - 1, \ldots, p_k - 1\right) $, we have
\begin{align*}
L_k  &= \lcm \left( L_2, p_3 - 1, \ldots, p_k - 1\right) \\
 &\leq 2^{2-k} \, L_2 \, (p_3 - 1) \, (p_k - 1), \quad \text{since all terms are even,}\\
 & < 2^{2-k} \, L_2 \, p_3 \cdots p_k.
\end{align*}
Applying Lemma \ref{lem:mu_Wief}, we get
\begin{align*}
g\left( 1,2^{L_k}-1,2^{L_k}\right) &> \log\left( \frac{p_1 \cdots p_k}{2^{k+1} L_k}\right) \\
 &> \log\left( \frac{p_1 p_2}{8 L_2}\right) = 2.278\ldots.
\end{align*}
Thus, we obtain an infinite sequence of $\mu$-hits of the form $\left( 1,2^{L_k}-1,2^{L_k}\right)$ for $k=2,3,\ldots$. However, the exponents $L_k$ are not necessarily distinct and it can happen that $L_{k+1} = L_k$ for some $k$. To show that we obtain  infinitely many distinct triples, it suffices to observe that $L_k \geq p_k - 1$, so that the sequence of powers $2^{L_k}$ is unbounded.
\end{proof}

\begin{conjecture}\label{conj_mu}
There are infinitely many $\mu$-hits.
\end{conjecture}

Finally, we formulate the above conjecture that reflects the preceding arguments and findings. Although our results need not assume its validity, it is important in view of Theorem \ref{th:mu_ABC_LBH}.

\section{Conclusion}

Somehow, the results presented so far show that the Collatz and $abc$ conjectures are more connected than previously thought and suggest that they may well be of comparable hardness. Nevertheless, the validity of the $abc$ conjecture is very unlikely to imply that of the Collatz conjecture. Indeed, the general formulas linked to $T$ iterations involve an unbounded number of terms which are mostly sums of powers of 2 and 3. Therefore, a generalization of the $abc$ conjecture might be required to treat expressions with much more terms than just $a,b,c$. Various statements have been proposed (see \cite{Bro94,Mar16}), but they are not well suited in this context. Further studies are necessary to strengthen the link between these not-so-distant conjectures and investigate how far the notion of $\mu$-hit can contribute.

\section*{Acknowledgements}
The author is grateful to Nik Lygeros for helpful comments.

\hspace{10pt}

\textsc{Institut de Physique du Globe de Paris, Université Paris Cité, France}

{\it E-mail: }{\tt rozier@ipgp.fr}

\end{document}